\documentclass[11pt,letterpaper]{amsart}
\usepackage{amsmath,amssymb,amsthm}
\usepackage{enumerate}
\usepackage{url}

\DeclareMathOperator{\Tr}{Tr}
\DeclareMathOperator{\zech}{Zech}

\renewcommand{\phi}{\varphi}
\newcommand\Weil[2]{{\rm W}_{#1,#2}}

\newcommand{\F}{{\mathbb F}}
\newcommand{\C}{{\mathbb C}}
\newcommand{\Q}{{\mathbb Q}}
\newcommand{\Z}{{\mathbb Z}}
\newcommand{\Fp}{\F_p}
\newcommand{\Ft}{\F_2}

\newcommand{\tfr}[1]{{\widehat{#1}}}
\newcommand{\tfrb}[1]{\widehat{#1}_{b}}

\newcommand{\card}[1]{\left|{#1}\right|}

\newcommand\Gauss[1]{{\tau}_{#1}}
\newcommand\abs[1]{{\vert {#1}\vert}}

\newcommand\val[1]{{\rm V}_{#1}}
\newcommand\valp{{\rm val}_{p}}
\newcommand\grmul[1]{{#1}^{\times}}

\newtheorem{theorem}{Theorem}[section]
\newtheorem{lemma}[theorem]{Lemma}
\newtheorem{proposition}[theorem]{Proposition}
\newtheorem{corollary}[theorem]{Corollary}
\newtheorem{conjecture}[theorem]{Conjecture}
\theoremstyle{definition}
\newtheorem{definition}[theorem]{Definition}
\newtheorem{remark}[theorem]{Remark}
\newtheorem{problem}[theorem]{Problem}

\def\Qz{{\Q}(\zeta_p)}
\def\galz{{\rm Gal}\big(\Qz\big)}

\title{New open problems related to old conjectures by Helleseth}

\author{Daniel J. Katz}
\address{Department of Mathematics, California State University, Northridge, \: United States}

\author{Philippe Langevin}
\address{Institut de Math\'ematiques de Toulon, Universit\'e de Toulon, France}

\date{first version: 29 December 2014; this version: 23 April 2015}

\begin{document}

\begin{abstract} 
Recently, very interesting results have
been obtained concerning the Fourier spectra
of power permutations over a finite field. In
this note we survey the recent ideas of Aubry,
Feng, Katz, and Langevin, and
we pose new open problems related 
to old conjectures proposed by Helleseth 
in the middle of the seventies.
\end{abstract}

\maketitle

\section{Introduction}

Let $L$ be a finite field of characteristic $p$
and  order $q$. One defines the Fourier coefficient
of a mapping $f\colon L \to L$ at a point $a\in L$
as
\[
\tfr f(a) = \sum_{x\in L} \mu( f(x) - ax ),
\]
where $\mu\colon L \to \C$ is the canonical additive character.

Strictly speaking, $\tfr f(a)$  is the Fourier 
coefficient of the complex map $\mu\circ f$ at 
the additive character $\mu_a\colon x\mapsto \mu(a x)$.
The minus sign that appears in the definition
of the Fourier coefficient is not usual, but 
there are several good reasons to adopt it. Above
all, one should note that $\mu_a$ will be an eigenvector
of eigenvalue $\tfr f(a)$ for the operator 
of convolution by $\mu \circ f$ over the mappings from $L$
to $\C$. 

In this paper, we are mainly interested in the
Fourier coefficient of the power mapping $f\colon x\mapsto x^s$
where $s$ is a positive integer. In that case,
the Fourier coefficient is sometimes called a \emph{Weil sum},
and we also use the notation
\[
\Weil Ls(a) = \tfr f(a) = \sum_{x\in L} \mu( x^s - ax ).
\]

In the case  where the exponent $s$ is coprime to $q-1$,
we say that it is an {\it invertible exponent}, and the mapping $f\colon x\mapsto x^s$ is  called a {\it power permutation},
because it is indeed a permutation of $L$.
In this case, the Fourier coefficient at the origin 
is equal to zero:
\[
\Weil Ls(0) = \tfr f(0) = \sum_{x\in L} \mu( x^s )
= \sum_{x\in L} \mu( x ) 
= 0.
\]

A power permutation  $f$, or its exponent $s$, is 
said to be \emph{singular} if there exists an $a\in\grmul L$ 
such that $\tfr f(a) = 0$.  We now present 
the first conjecture  proposed in 1976 by Helleseth \cite{TOR}. 
 
\begin{conjecture}[Helleseth Vanishing Conjecture, 1976]
\label{HVC} 
If $|L|>2$ and $s\equiv 1\pmod{p-1}$, then $s$ is singular.
\end{conjecture}

The conjecture was based on numerical evidence. The Fourier 
spectra of all the exponents over the extensions of $\F_2$ with 
degree less than or equal to $25$ have been computed in 2007 by Langevin 
\cite{NPP}, and no counterexample was found. Up to now, very little in the way of partial results has been obtained for this conjecture.  This question has every appearance 
of difficulty.
Let us consider the special case when the exponent $s=q-2$ satisfies Conjecture \ref{HVC}.
Then $s=q-2 \equiv -1 \pmod{p-1}$, so the congruential hypothesis of Conjecture \ref{HVC} is satisfied for $s=q-2$ if and only if $-1\equiv 1 \pmod{p-1}$, which is true if and only if $p=2$ or $p=3$.
For such an exponent, the Fourier coefficient 
is one plus a Kloosterman sum 
$$
\Weil L{q-2} ( a ) = 1 + \sum_{x \in\grmul L} \mu\left( \frac 1x - ax \right).
$$
If $L$ is of characteristic $2$, one can use the theory of elliptic curves as in \cite{LW} to prove that $\Weil L{q-2}(a)$ assumes all integer values
divisible by $4$ in the range $[1-2\sqrt q, 1+2\sqrt q]$ as $a$ runs through $\grmul L$,
a consequence of Deuring's work. In particular, the
Fourier coefficient $\Weil L{q-2}(a)$ takes the value $0$ for some $a\not=0$. 
In characteristic $3$, it was shown (see \cite{KL}) that $\Weil L{q-2}(a)$ assumes all integer values divisible by $3$ in the range $[1-2\sqrt q, 1+2\sqrt q]$ as $a$ runs through $\grmul L$.
\begin{problem}
Prove the Helleseth Vanishing Conjecture for $p\in\{2,3\}$ and
$s=q-2$ without using of the theory of elliptic curves. 
\end{problem}

\begin{remark}
In characteristic $p>3$ the exponent $-1$, i.e., the
exponent $q-2$, is not singular. Indeed, for such fields $\Weil L{q-2}(a)$ has been shown to be nonzero when $a\not=0$ in \cite{KRV}.
\end{remark}

An exponent $s$ is said to be {\it $r$-valued} if the number of distinct 
Fourier coefficients on $\grmul L$ is $r$. The following theorem is
a consequence of recent results of Feng and Katz.
\begin{theorem}\label{Thomas} If $[L:\Ft]$ is a power of
two, then an invertible exponent is not three-valued.
\end{theorem}
Feng \cite{FENG} proved the above theorem assuming that 
at least one of the Fourier coefficients is zero on $\grmul L$.
Katz \cite{DKC} proved that this vanishing always 
occurs when the exponent $s$ is three-valued.
Katz's methods work in arbitrary characteristic, but Feng's work was specific to characteristic $2$, so that Theorem \ref{Thomas} only concerns fields of characteristic $2$.
More recently, Katz \cite{DKDIV} has shown that one can replace $\Ft$ by $\F_3$ in Theorem \ref{Thomas}.
The proof of Theorem \ref{Thomas} and its analogue in characteristic $3$ verify parts of a conjecture by Helleseth, who proposed that the result should hold in arbitrary characteristic.
\begin{conjecture}[Helleseth Three-Valued Conjecture, 1976] 
\label{HSC}
For any prime $p$, if $[L:\Fp]$ is a power of
two, then the spectrum of an invertible exponent
is not three-valued.
\end{conjecture}
This conjecture remains open for $p>3$.

In this paper, we survey old and recent results 
dealing with Fourier coefficients of power 
mappings to conclude with a very interesting 
open question in the theory of finite fields.

\section{Fourier Coefficients and Convolution}

The \emph{Fourier coefficient} at a point $a\in L$ of 
a complex function $F$ defined over $L$ is
\[
\tfr F(a) = \sum_{x\in L} F(x) \bar\mu( ax ).
\]
\begin{remark} 
$\tfr F(a)$ is a scalar product. The additive characters 
of $L$ form an orthogonal basis of the complex mappings 
on domain $L$.
\end{remark}
If $f\colon L \to L$, then we write $\tfr f(a)$ in place of $\tfr{\mu\circ f}(a)$ by a common abuse of notation.
As mentioned in the Introduction, the Weil sum $\Weil Ls(a)$ 
is nothing but the Fourier coefficient $\tfr f(a)$ of
$\mu\circ f$ when $f(x)=x^s$.

If $f \colon L \to L$ and $b \in L$, then we write $f_b$ for the function $b f$, that is, $f_b(x)=b f(x)$ for $x \in L$.
In certain applications it is important to understand the Fourier transform not only of $f$, but of all scalar multiples $f_b$ of $f$.
Power permutations are special in this regard: if $f(x)=x^s$ with $\gcd(s,q-1)=1$ and $b \in \grmul L$, then
\begin{equation}
\label{INVMUL}
\tfrb f(a) 
= \sum_{x\in L} \mu(b x^s - a x)
= \sum_{x\in L} \mu((b^{1/s} x)^s - a b^{-1/s} (b^{1/s} x))
= \tfr f(ab^{-1/s}),
\end{equation}
where $1/s$ is interpreted modulo $q-1$.
Also, when $f(x)=x^s$ is a power permutation, $\tfrb{f}(a)$ is always a real number:
\begin{equation}\label{Larry}
\overline{\tfrb f(a)} = \sum_{x\in L} \mu(-(b x^s - a x))= \sum_{x\in L} \mu(b (-x)^s - a (-x))=\tfrb f(a),
\end{equation}
since the condition $\gcd(s,q-1)=1$ makes $s$ odd when the characteristic of the field is odd.

Fourier coefficients satisfy general rules, namely, the
inversion formula,
\[
\sum_{a\in L} \tfr F(a) \mu(ax) = q F(x),
\]
or more generally the Poisson formula over an additive subgroup $S$ of $L$,
\[
\sum_{a\bot S} \tfr F(a) \mu(ax) = \frac q{\card S} \sum_{s\in S} F(x+s),
\]
where $S$ is considered a subspace of the $\F_p$-linear space $L$ equipped with inner product $(x,y) \mapsto \Tr(x y)$, with $\Tr\colon L \to \F_p$ the absolute trace.
There is also the Parseval-Plancherel identity,
\[
\sum_{a\in L} \abs{\tfr F(a)}^2 = q \sum_{x\in L} \abs{F(x)}^2. 
\]
In this context, one introduces the convolutional product
of two complex mappings $F$ and $G$ at $z\in L$,
\[
F \ast G ( z ) = \sum_{x+y=z} F(x) G(y).
\]
The $\C$-algebra of complex maps equipped with
this product is usually denoted by $\C[L]$.
For $b\in L$, we denote by $\delta_b$ the {\it Dirac function}
\[
\delta_b(x) = \begin{cases}
1 & \text{if $x=b$}, \\
0 & \text{otherwise},
\end{cases}
\]
and we refer to the set of functions $\{\delta_b\}_{b \in L}$ as the {\it Dirac basis}, since it is a basis of the $\C$-vector space $\C[L]$.
Then $\C[L]$ has $\delta_0$ for its unit element. The $k$th power
of convolution is  
\[
F^{[k]}(z) = \sum_{x_1+\cdots+x_k = z} F(x_1) F(x_2)\cdots F(x_k),
\]
and one has the well known trivialization formulas
\begin{align}
\label{TRIV}
\tfr{ F\ast G} (a) & = \tfr F (a) \tfr G(a), \\
q \tfr{ FG} (a) & = \tfr F \ast \tfr G(a).\nonumber
\end{align}
\begin{remark}
The inversion formula and trivialization formulas
show that the Fourier transform is
a $\C$-algebra isomorphism from $\C[L]$ into $\C^L$ with
$$
   \tfr {\delta_0} = 1,\quad \tfr 1 = q\delta_0.
$$ 
\end{remark}
In the Dirac 
basis $\{\delta_b\}_{b\in L}$,

\[
\delta_b \ast F (t)= \sum_{y+x = t} \delta_b(y) F(x) = F(t-b),
\]
whence
\[
	\delta_b \ast F = \sum_{t} F(t-b) \delta_t,
\]
and so, since $\{\mu_a\}_{a \in L}$ is an eigenbasis for convolution by $F$ in the $\C$-linear space $\C[L]$, with $\tfr F(a)$ the eigenvalue for $\mu_a$, we have
\[
    \prod_{a\in L} \tfr F(a) = \det [ F( a - b) ]_{a,b\in L}.
\]
Then the Helleseth Vanishing Conjecture (Conjecture \ref{HVC}) is equivalent to the following claim.
\begin{conjecture}[Helleseth's 1976 Vanishing Conjecture, restated]
Let $s\equiv 1\pmod{p-1}$ be an integer coprime to
$q-1$. The rank of the matrix $[\mu\big((x-y)^s\big)]_{y,x \in L}$
is less than $q-1$.
\end{conjecture}

For $f\colon L \to L$, the product
$D(f)= \prod_{a\in\grmul L} \tfr f(a)$ of the Fourier coefficients on $\grmul L$ appears
naturally in two ways. Firstly, by considering the 
convolution by the mapping $x\mapsto \mu\big(f(x)\big) - 1$,
one can show that
$$
	-q D(f) = \det [ \mu\big(f(x-y) \big) - 1]_{x,y\in L}.
$$

Secondly, if $f(x)=x^s$ for an invertible exponent $s$, the number of solutions in $L^n$ of 
of
\begin{equation}
\label{LSYS}
\begin{cases}
f(x_1) + f(x_2) + \ldots + f(x_n) & = 0,\\
\lambda^1 x_1 +
\lambda^2 x_2 + \cdots +
\lambda^n x_n & = 0,
\end{cases}
\end{equation}
where $\lambda$ has order $q-1$ in $\grmul L$,
can be written as  
\begin{align*}
\frac 1{q^{2}}\sum_{a,b} \sum_{x_1,\ldots,x_n} 
 \mu_b\left( \sum_{i=1}^nf(x_i) \right)
 \bar\mu_a \left(\sum_{i=1}^n\lambda^ix_i \right)
&=
\frac 1{q^{2}}\sum_{a,b}
\prod_{i=1}^n \tfrb f(a\lambda^i) \\
& =q^{n-2}
+ 
\left(\frac {q-1}{q^2}\right) \sum_{j=0}^{q-2} D_{j,n}(f),
\end{align*}
where $D_{j,n}(f) = \prod_{k=j}^{j+n-1} \tfr f( \lambda^{k})$, which equals $D(f)$ when $n=q-1$. 
In particular, as remarked by Helleseth \cite{TOR-SETA}, the Vanishing
Conjecture is equivalent to saying that the number of solutions 
of the system \eqref{LSYS} with $n=q-1$ is equal to $q^{q-3}$.
\section{Spectrum of a Power Mapping}\label{Joseph}

Let $F\colon L \to \C$. The set of 
the Fourier coefficients is called the \emph{spectrum} of $F$. The
set of  Fourier coefficients on $\grmul L$ is called the \emph{reduced
spectrum} of $F$. 

The values and the multiplicities of 
the Weil sums $W_{L,s}(a)$ of exponent $s$ do not change if 
we replace $s$ by $ps$ or $1/s$ (modulo $q-1$).
We write  $s^\prime \sim s$  if there exists 
$j$ such that $s^\prime \equiv p^j s\pmod{q-1}$. We
say that the exponents $s'$ and $s$ are {\it equivalent} and write $s^\prime \approx s$ if $s^\prime\sim s$ or  
$s^\prime\sim 1/s$.

Recall from \eqref{INVMUL} that if $f(x)=x^s$ with $\gcd(s,q-1)=1$ and $b\not=0$, then
\[
\tfrb f(a) 
=\tfr f(ab^{-1/s}),
\]
where $1/s$ is interpreted modulo $q-1$.
Thus $f_b=b f$ has the same spectrum (reduced or not) 
as $f$ for all $b\in\grmul L$. Note that a general
power mapping does not satisfy this property, with a simple example being $x\mapsto x^{q-1}$ when $q>2$. For
future work, it could be important to know more about 
the mappings satisfying this  spectrum invariance property. 

\begin{problem}[invariance]
Find necessary and sufficient conditions for a map (or permutation) $f \colon L\rightarrow L$ to have the property that the spectrum of $f$ is equal to the spectrum of $b f$ for all $b\in\grmul L$.
\end{problem}

Let $\zeta_p = \exp( 2i\pi/p)$. Then let
$\wp=(1-\zeta_p)$, the prime ideal above  $p$ in $\Z[ \zeta_p ]$.
For a power permutation 
$f$, the $\wp$-divisibility of the Fourier coefficient follows
from
\begin{equation}\label{Peter}
	\tfr f(a) \equiv \tfr f(0) \equiv 0 \pmod {\wp},
\end{equation}
since $1=\mu_0(x) \equiv \mu_a(x) \pmod{\wp}$ for every $x \in L$.

Another important fact satisfied by power permutations 
is the invariance of the spectrum under the action of
the Galois group of $\Qz$. Indeed, considering the element
$\phi_r$ in $\galz$ that maps $\zeta_p$ to $\zeta_p^r$,
one has
$$
\phi_r ( \tfr f(a) ) = \sum_{x\in L} \mu( r f(x) - arx ) = \tfr f_r(a r) = \tfr f( a r^{1-1/s} ).
$$

\begin{lemma}[algebraic degree]
The spectrum of a power permutation of exponent
$s$ has all values in $\Z$ if and only if $s\equiv 1\pmod{p-1}$. 
If $d \mid p-1$, then the Fourier coefficients reside in the degree $d$ 
extension of $\Q$ lying within $\Q(\zeta_p)$ if and only if 
$s\equiv 1\pmod{(p-1)/d}$.
\end{lemma}
\begin{proof}
The first part appears in Helleseth's paper \cite[Theorem 4.2]{TOR}, 
and is a consequence of the second part.
Let $d \mid p-1$, and let $r$ be an element of multiplicative order
$(p-1)/d$ in $\grmul \F_p$.  Then for $f(x)=x^s$, we see that $\phi_r \big(\tfr f(a) \big) = \tfr f(a)$ for all $a \in L$ if and only if $\tfr f(a r^{1-1/s} ) = \tfr f(a)$ for all $a \in L$.  By Fourier inversion, the latter is true if and only if $\mu \circ f( r^{1/s-1} x) = \mu \circ f(x)$ for all $x \in L$, which in turn is true if and only if $\mu ( r^{1-s} x^s ) = \mu ( x^s )$ for all $x \in L$, which happens if and only if $r^{1-s} = 1$, i.e., if and only if $(p-1)/d$ divides $1-s$.
\end{proof}
Assume that $f(x)=x^s$ with $s$ an $r$-valued exponent with values $A_1$, $A_2$, $\ldots$, $A_r$, and denote by $\sigma_i$ the $i$th signed elementary symmetric function of these values, that is,
$$
\sigma_0 = 1,\quad \sigma_1 = -\sum_{i=1}^r A_i,\quad\ldots\quad,\quad \sigma_r = (-1)^r\prod_{i=1}^r A_i.
$$ 
For all $a$ in $L$, we have
\[
           \sum_{i=0}^r \sigma_i \tfr f(a)^{r-i} =  \sigma_r \delta_0(a).
\]
Denote the $n$th convolutional power of $\mu\circ f$ by
$$
f^{[n]}(z) = \sum_{x_1+\ldots+x_n = z} \mu( f(x_1)+\ldots+f(x_n)).
$$
Then for all $z$ in $L$, we have
\[
           q \sum_{i=0}^r \sigma_i(A_1,\ldots,A_r) f^{[r-i]}(z) = \sigma_r(A_1,\ldots,A_r).
\]
In particular, $q$ divides $\prod_{i=1}^r A_i$.  Recall from Conjecture \ref{HVC} that this product is actually conjectured to be $0$ when $s \equiv 1 \pmod{p-1}$.

For any $f\colon L \to L$, the number of solutions in $L^n$ of the system
\begin{align*}
x_1 + x_2 + \cdots + x_n & = u \\
f(x_1) + f(x_2) + \ldots + f(x_n) & = v 
\end{align*}
is given by 
\begin{align*}
N(u,v) &= \frac 1{q^{2}}\sum_{a,b} \sum_{x_1,\ldots,x_n} 
 \mu_b\big( \sum_{i=1}^nf(x_i) - v  \big)
 \bar\mu_a( \sum_{i=1}^n x_i   - u     )\\
&=
\frac 1{q^{2}}\sum_{a,b} \tfrb f(a)^n \mu( au - b v ).
\end{align*}
In \cite{AL}, Aubry and Langevin used
this relation and the little Fermat theorem
to obtain the following  congruence result.
\begin{theorem}[Aubry, Langevin, 2013] 
Let $L$ be a finite field of order $q>2$. 
If $f$ is a power permutation of $L$ of exponent 
$s \equiv 1\pmod{p-1}$, then there is an $a\in \grmul L$ such
that $\tfr f(a) \equiv 0 \pmod 3$.
\end{theorem}
\begin{problem}
Is it possible to obtain such a divisibility result
involving another prime $\ell \not=p$?
\end{problem}
The following result \cite[Theorems 1.7, 1.9]{DKC} proved the vanishing 
of a Fourier coefficient on $\grmul L$ that finished the proof of Theorem \ref{Thomas}.
\begin{theorem}[Katz, 2012]\label{Katherine}
If $x^s$ is a three-valued power permutation, then $s$ is singular, $s \equiv 1\pmod{p-1}$, and the spectrum $\{0, A, B\}$ of $x^s$ lies in $\Z$.
\end{theorem}
The action of the Galois group of the cyclotomic field is the main ingredient of the proof. 
\begin{problem}
Find an analogue of Theorem \ref{Katherine} for four-valued exponents.
\end{problem}
When $f(x)=x^s$ is a three-valued power permutation with values $0$, $A$, and $B$, the number of solutions in $L^2$ of the system
\begin{align*}
  x + y & = 1 \\
  x^s + y^s & = 1 
\end{align*}
is known (e.g., see \cite[Lemma 4.2]{DKDIV}) to be
\begin{equation}
\label{NONE}
	V = N(1,1) = A + B - \frac {AB} q.
\end{equation}
The following relation between the third power moment of the Fourier coefficient and $N(1,1)$ is equivalent to an observation of Blokhuis and Calderbank \cite{Blokhuis-Calderbank} about the weight distribution of certain cyclic codes:
\begin{align}
  \label{Elaine}
  \sum_{a\in L} \tfr f(a)^3
&= \sum_{x,y,z \in L} \mu( x^s+y^s+z^s ) \sum_{a \in L} \mu_a(x+y+z) \\
&= q \sum_{x+y+z=0} \mu( x^s+y^s+z^s ) \nonumber \\
&= q \sum_{x+y=0} \mu( x^s+y^s ) +  q \sum_{z\not=0} \sum_{x+y+z=0} \mu( x^s+y^s+z^s ) \nonumber \\
&= q \sum_{x} \mu(x^s+(-x)^s) + q \sum_{z\not=0} \sum_{X+Y+1=0} \mu( (X^s+Y^s+1) z^s ) \nonumber \\
&= q^2 + q (q-1) V - q (q-V) \nonumber \\
&= q^2 V. \nonumber
\end{align}
In the penultimate equality, we use the fact that the condition $\gcd(s,q-1)$ makes $s$ odd when the characteristic of $L$ is odd (and so $N(1,1)=N(-1,-1)$).

More generally, the product of nonzero
spectral values $A_1$, $A_2$, $\ldots$, $A_n$
appears naturally by Fourier analysis.
Define the signed elementary symmetric functions
$$
\sigma_0 = 1,\quad \sigma_1 = -\sum_{i=1}^n A_i,\quad\ldots\quad,\quad \sigma_n = (-1)^n\prod_{i=1}^n A_i.
$$ 
Let
us consider the polynomial 
$$
P(T) = \prod_{i=1}^n (T-A_i)=\sum_{i=0}^n \sigma_{n-i} T^{i}. 
$$
The rule of trivializations \eqref{TRIV} shows that
$$
	\sum_{i=0}^n \sigma_{n-i} F^{[i]} \ast F = 0.
$$
This means  that 
$\sum_{i=0}^n \sigma_{n-i} F^{[i]}$ 
is in the kernel of convolution by $F$.
Recall that the characters $\{\mu_a : a \in L\}$ form an eigenbasis for convolution by $F$ in the $\C$-linear space $\C[L]$, with $\tfr F(a)$ the eigenvalue for $\mu_a$.
If we let $Z=\{a \in L: \tfr F(a)=0\}$, then $\{\mu_c : c \in Z\}$ is a basis of the nullspace for convolution by $F$, and we can write
$$
\sum_{i=0}^n \sigma_{n-i} F^{[i]}
= \sum_{c\in Z} \lambda_c \mu_c,
$$
for some coefficients $\lambda_c \in \C$.
Taking the Fourier coefficient at $c \in Z$, we
obtain with $\lambda_c=\frac{(-1)^n}q \prod_{i=1}^n A_i$.
\begin{problem}[product]
For which $s$ is the product of non-zero spectral
values $\prod_{i=1}^{n} A_i$ of $F(x)=\mu(x^s)$ divisible by $q$?
\end{problem} 

\section{$p$-Divisibility}

For a power permutation $f(x)=x^s$, we define 
$$
\val L(s) = \min_{a\in L} \valp \big( \tfr f(a) \big).
$$
This minimum valuation is deducible from Stickelberger's Theorem on the $p$-divisibility of the Gauss sum
\[
\Gauss L (\chi) = \sum_{a \in \grmul L} \mu(a) \chi(a),
\]
for $\chi$ a multiplicative character of $L$.
One has \cite[eq.~(3)]{AKL} the formula
\begin{equation*}
\tfr f(a) 
= \frac q{q-1} + \frac 1{q-1} \sum_{\chi\not=1} 
\Gauss L (\chi) \Gauss L(\bar\chi^{s}) \chi^s(-a),
\end{equation*}
whence \cite[Lemma 4.1]{AKL}
$$
	\val L( s ) = \min_{1\not=\chi\in\grmul L} \valp\big(\Gauss L (\chi) \Gauss L(\bar\chi^{s})\big).
$$
Using the Hasse-Davenport relation, given an extension $L/K$, we obtain \cite[Corollary 4.2]{AKL}
\begin{equation}\label{Nestor}
\val L (s) \leq  \val K(s) \times [L:K].
\end{equation}

Recall the Helleseth Three-Valued Conjecture (Conjecture \ref{HSC}), which states that if $[L:\Fp]$ is a power of two, then the spectrum of a power permutation is not three-valued.
Feng \cite[Theorem 2]{FENG} showed that this conjecture holds in characteristic $p=2$ under the additional assumption that at least one Fourier coefficient in the reduced spectrum vanishes.
In \cite[Corollary 1.10]{DKC}, \cite[Theorem 1.7]{DKDIV} Katz showed that the conjecture holds in characteristic $p=2$ and $3$ without additional assumptions.

Feng uses the following proposition \cite{Calderbank-McGuire-Poonen-Rubinstein}
to obtain Conjecture \ref{HSC} in even characteristic
under the assumption that the exponent is singular.
\begin{proposition}[Calderbank, McGuire, Poonen, Rubinstein, 1996]
Let $s\not\approx 1$ be an invertible exponent. 
If $[L:\F_2]$ is a power of two, then 
$$2 \times \val L(s) \leq [L:\F_2].$$
\end{proposition}
\begin{remark}\label{Harold}
In fact, if $s\not\approx 1$ is an invertible exponent and $[L:\F_p]$ is a power of two, then $2 \times \val L(s) \leq [L:\F_p]$ still holds for an arbitrary prime $p$, as we shall show below.
\end{remark}
The following result \cite[Corollary 4.4]{AKL} shows what happens when $L/K$ is a quadratic extension in which $s \equiv 1 \pmod{|\grmul K|}$ but $s\not\equiv 1 \pmod{|\grmul L|}$.
The characteristic $2$ case was proved by Charpin \cite[Theorem 1, Corollary 1]{Charpin}.
\begin{lemma}[quadratic extension]
\label{QUADRA}
Let $L/K$ be a quadratic extension. If $x^s$ is
constant over $\grmul K$ but not over $\grmul L$, then there is an $a \in \grmul L$ such that $\tfr f(a) =- \abs{K}$ and $2 \times \val L(s) = [L:\F_p]$.
\end{lemma}
Using \eqref{Nestor}, we now see that if $1\not\approx s\equiv 1 \pmod{p-1}$ and $[L:\F_p]=2^r$, then $ 2 \times \val L(s) \leq [L:\F_p]$, thus validating Remark \ref{Harold}.

\section{Differential Multiplicity and the Uniformity Property}

For $f(x)=x^s$ a power permutation over $L$, we denote by $N(u,v)$ the number of
solutions in $L^2$ of the system
\begin{align*}
      x    + y    & = u\\
      f(x) + f(y) & = v.
\end{align*}
Since $(-y)^s=-y^s$, $N(u,v)$ is also 
the number of solutions of $x-y=u$, $f(x)-f(y) = v$.
Therefore, the numbers $N(u,v)$ are called the {\it differential multiplicities} 
of the exponent $s$.

Recall that $N(1,1)$ arose in equation \eqref{NONE} and was subsequently shown to be connected to the third power moment of the Fourier coefficients of $f$.
Note that if $u\not=0$, then $N(u,v)=N(1,v/u^s)$.
\begin{definition}
We say that a power permutation $f$ 
is $\Delta$-uniform over $L$  if 
the number of solutions $N(1,v)$ in $L^2$ of the system
\begin{align*}
      x   + y     & = 1,\\
      f(x)+  f(y) & = v,
\end{align*}
is equal to $0$ or $\Delta$ for all $v\not=1$.
\end{definition}

Katz \cite[Lemma 4.4, Theorem 4.5, Remark 4.6]{DKDIV} proved the following theorem using the group algebra techniques of Feng.
\begin{theorem}[Katz]
\label{UNIFORM}
Let $f(x)=x^s$ be a power permutation over $L$ with three-valued spectrum $\{0,A,B\}$, and write $A=p^a\alpha$, $B=p^b\beta$ and $A-B=p^c\gamma$ where $p\nmid \alpha, \beta, \gamma$.
Then $\alpha\beta\gamma$ divides the differential multiplicities $N(u,v)$ for all $v\not=u^s$ and 
$$
\abs{\alpha\beta\gamma} \leq - \frac{AB}q,
$$
leading to the alternative:
\begin{enumerate}[(i).]
\item\label{Alice} $a, b > \frac{1}{2} [L:\F_p]$ (impossible when $[L:\F_p]=2^r$), or
\item\label{Bob} $a=b = \frac{1}{2} [L:\F_p]$, $|\gamma|=1$, and $s$ is a $\abs{\alpha\beta}$-uniform exponent.  
\end{enumerate}
\end{theorem}
\begin{remark}\label{Raphael}
Note that because of Lemma \ref{QUADRA}, 
case \eqref{Alice} is impossible when $[L:\F_p]$ is
a power of $2$.
\end{remark}
\begin{corollary}\label{Richard}
If $p=2$ or $p=3$ then the Helleseth Three-Valued Conjecture is true.
\end{corollary}
\begin{proof}
$N(1,1)$ is divisible by $p$ when $p=2$ (see \cite[Proof of Theorem 2]{FENG}) and when $p=3$ (see \cite[Lemma 4.2]{DKDIV}).
If $f(x)=x^s$ is three-valued, then $A$ and $B$ lie in $\Z$ (by Theorem \ref{Katherine}), and so they are divisible by $p$ by \eqref{Peter}, and then \eqref{NONE} tells us that $p q \mid A B$, so that we cannot be in case \eqref{Bob} of Theorem \ref{UNIFORM}.
On the other hand, Remark \ref{Raphael} shows that case \eqref{Alice} is impossible when $[L:\F_p]$ is a power of $2$.
\end{proof}
We end this section by proving Theorem \ref{UNIFORM}
without the language of the group algebra. Let $f(x)=x^s$ be a power 
permutation with a three-valued spectrum with values $0$, $A$, and $B$.
In view of \eqref{INVMUL}, for each $u \in \grmul L$, the spectrum of $f_u$ has the same values $0$, $A$, and $B$.
Write
\[
   A = p^a \alpha,\quad B=p^b \beta,\quad A-B=p^c \gamma
\]
with $\alpha$, $\beta$ and $\gamma$ coprime to $p$.
\begin{lemma}\label{Laurent}
The integers $\alpha$, $\beta$ and $\gamma$ are
pairwise coprime, $\alpha \gamma$ divides $q-B$, and
$\beta\gamma$ divides $q-A$.
\end{lemma}
\begin{proof}
  From the first and second power moments of the spectrum of $f$, one may deduce (see \cite[Proposition 3.2]{DKC}) that if $\tfr f(x)=A$ for $N_A$ values of $x \in \grmul L$ and $\tfr f(x)=B$ for $N_B$ values of $x \in \grmul L$, then
  \begin{align*}
    N_A & = \frac{q(q-B)}{A(A-B)} \\
    N_B & = \frac{q(q-A)}{B(B-A)},
  \end{align*}
  from which all of our claims quickly follow.
\end{proof}

Now we proceed to a proof of Theorem \ref{UNIFORM}.
Using a character counting principle, we have
\begin{align}
   N(u,v) &=\frac 1{q^2} \sum_{y,z}\sum_{w, x\in L} \mu_x( f(y)+f(z) - v ) \bar\mu_w(y+z - u ) \label{Gabriel} \\ 
&=\frac 1{q^2} \nonumber
\sum_{w,x\in L} \tfr {f_x}(w)^2 \bar\mu_x(v) \mu_w( u ).
\end{align}
In other words, $\tfr {f_x}(w)^2$ is a Fourier coefficient
of $N$ over the group $L\times L$. On the other hand,
\begin{align}
\frac{1}{q^2} \sum_{w,x\in L} \tfr {f_x}(w) \bar\mu_x(v) \mu_w( u ) 
& = \frac{1}{q} \sum_{x\in L} \mu_x( f(u) ) \bar\mu_x(v) \label{Veronica} 
 \\
& = \delta_{v}(f(u)) \nonumber \\
& = \delta_v(u^s). \nonumber
\end{align}
By subtraction of $A$ times \eqref{Veronica} from \eqref{Gabriel}, it follows that $N(u,v)$ is divisible by $\beta\gamma$ when $v\not=u^s$, and similarly, by subtracting $B$ times \eqref{Veronica} from \eqref{Gabriel}, it follows that $N(u,v)$ is divisible by $\alpha\gamma$ when $v\not=u^s$.
So when $v\not=u^s$, Lemma \ref{Laurent} shows us that $\alpha\beta\gamma\mid N(u,v)$.

Applying the Parseval relation to the mapping $N$,
\begin{equation}\label{Roger}
	\sum_{w, x} \tfr {f_x}(w)^4 = q^2 \sum_{u,v} N(u,v)^2,
\end{equation}
where we use $\tfr{f_x}(w)^4$ in place of $|\tfr{f_x}(w)|^4$ because \eqref{Larry} shows that $\tfr{f_x}(w)$ is always a real number.
For our power permutation $f$, recall that for any $x \in \grmul L$, the function $f_x=x f$ has the same spectrum as $f$, and note that $\tfr f(0)=0$.
On the other hand $f_0=0 f=0$ has spectrum $\tfr{f_0}(0)=q$ and $\tfr{f_0}(w)=0$ for $w \in \grmul L$.
        Also note that $N(u,v)=N(1,v/u^s)$ and $N(u,0)=0$ when $u\not=0$, while $N(0,0)=q$ and $N(0,v)=0$ when $v\not=0$.
        With these observations, \eqref{Roger} becomes
\begin{equation}
\label{CV}
	\sum_{w} \tfr f(w)^4 = q^2 \sum_{v\not=0} N(1,v)^2,
\end{equation}
and since the spectrum is three-valued with values $0$, $A$, and $B$, we have
\begin{align*}
        \sum_{w} \tfr f(w)^4 
& = (A+B) \sum_{w} \tfr f(w)^3 -AB \sum_{w} \tfr f(w)^2 \\
        & = (A+B) q^2 V - AB q^2,
\end{align*}
where we have used the calculation \eqref{Elaine} for the third power moment, and the well-known value $q^2$ of the second power moment (which can be obtained by a similar, but easier calculation).
We substitute the fourth power moment into \eqref{CV} to obtain
$$
	-V^2 + (A+B) V - AB  = \sum_{v\not\in\{0,1\} } N(1,v)^2,
$$
        because $V$ denotes the same number as $N(1,1)$.  Then since $\alpha\beta\gamma$ divides $N(1,v)$ when $v\not=1$, we have
        \begin{equation}\label{Hubert}
        N(1,v)^2 \geq |\alpha\beta\gamma| N(1,v)
        \end{equation}
        for all $v\not=1$, and so
\begin{align}
\label{FIRST}
       - (V-A) (V-B)  & \geq |\alpha\beta\gamma|  \sum_{v\not\in\{0,1\} } N(1,v) \\
 & = |\alpha\beta\gamma| (q-V), \nonumber
\end{align}
since $\sum_{v} N(1,v)=q$, inasmuch as it counts the solutions in $L^2$ of $x+y=1$, and we have noted that $N(1,1)=V$ and $N(1,0)=0$.
We substitute the value of $V$ from \eqref{NONE} into \eqref{FIRST}, and simplify to obtain
\[
	|\alpha\beta\gamma| \leq \frac {-A B}q.
\]
Note that this proves that $A$ and $B$ have opposite sign, and then
\begin{equation}
\label{INEQ}
|\gamma| \leq \frac{p^{a+b}}{q}.
\end{equation}

For the rest, we proceed as in the proof of Theorem 4.5 and Remark 4.6 in \cite{DKDIV}.
If $a\not=b$, then we would have $c=\min\{a,b\}$.
Then let $d=\max\{a,b\}$, and so $q |A-B|= q p^c |\gamma|  \leq p^{a+b+c} = p^{2 c+d}$.  But since $A$ and $B$ have opposite signs, $|A-B| > \max\{|A|,|B|\} \geq p^d$, so that $q p^d < q |A-B| \leq p^{2 c+d}$, and so $p^c > \sqrt{q}$, and thus $a, b > \frac{1}{2} [L:\F_p]$, which is case \eqref{Alice} in the statement of Theorem \ref{UNIFORM}.

If $a=b$, then inequality \eqref{INEQ} shows that $a=b \geq \frac{1}{2} [L:\F_p]$, and if this inequality is strict, we are again in case \eqref{Alice}.
Otherwise, we have $a=b=\frac{1}{2}[L:\F_p]$, and then $|\gamma|=1$ and inequality \eqref{INEQ} becomes an equality, as do the previous inequalities from which it was deduced, and in particular \eqref{Hubert} shows that $N(1,v)=0$ or $|\alpha\beta\gamma|=|\alpha\beta|$ for all $v\not=1$, that is, $s$ is $|\alpha\beta|$-uniform.  This is case \eqref{Bob} in the statement of Theorem \ref{UNIFORM}.

We have already noted why case \eqref{Alice} is impossible when $[L:\F_p]$ is a power of $2$ in Remark \ref{Raphael}.  This completes the proof of Theorem \ref{UNIFORM}.
\section{Conjectures on Differential Uniformity}

The previous section shows that the Helleseth Three-Valued
Conjecture depends on the nonexistence of uniform 
exponents in certain cases. We present a numerical experiment \cite{NICEXP}
which shows that the nonexistence of such exponents could be
the key point to obtain a proof of Helleseth's conjecture.

\begin{definition}
Let $s$ be an invertible exponent over the finite field $L$.
We say that $s$ is a \emph{nice exponent over $L$} if
the number $N(1,v)$ of solutions in $L^2$ of 
\begin{align*}
      x   + y    = 1,\\
      x^s + y^s  = v,
\end{align*}
takes at most $3$ values as $v$ runs through $L$.
\end{definition}
\begin{remark}
A $\Delta$-uniform exponent is nice.
\end{remark}
Recall the relation $\approx$ 
defined near the beginning of Section \ref{Joseph}.
\begin{remark}
If $s$ is a nice exponent, then every exponent
$s'\approx s $ equivalent to $s$ is also nice.
\end{remark}

\begin{remark}
If $s\approx 1$ then $s$ is trivially a nice exponent with two differential
multiplicities: $0$ and $q$. 
\end{remark}
We show two paradigmatic examples of nice exponents in the following propositions.
\begin{proposition}
\label{PARADIG}
The exponent $s=3$ is nice if and only if $q\not\equiv 1 \pmod 3$, with the following differential multiplicities:
\vskip 3mm
\begin{center}
\begin{tabular}{ccc}
\hline
characteristic & differential multiplicities & respective frequencies \\
\hline
\hline
$2$ & $0$, $2$ & $q/2$, $q/2$ \\
$3$ & $0$, $q$ & $q-1$, $1$ \\
$p>3$ & $0$, $1$, $2$ & $q/2-1$, $1$, $q/2-1$\\
\hline
\end{tabular}
\end{center}
\end{proposition}
\begin{proof}
The congruence condition is necessary and sufficient to make $3$ an invertible exponent, and the result for characteristic $3$ is trivial, for then $3\approx 1$.
In other characteristics, $N(1,v)$ is the number of solutions in $L$ of $x^3+(1-x)^3=v$, which is the number of roots in $L$ of the quadratic polynomial $x^2-x+(v-1)/3$.
In characteristic $2$, the additive Hilbert Theorem 90 shows that this quadratic polynomial has zero or two roots in $L$ depending on whether the absolute trace (from $L$ to $\Fp$) of $(v-1)/3$ is $1$ or $0$, respectively.
In odd characteristic, the quadratic polynomial has zero, one, or two roots in $L$
depending on whether the discriminant $(7- 4 v)/3$ is a quadratic
nonresidue, zero, or a quadratic residue in $L$, respectively.
\end{proof}
\begin{proposition}
\label{MINUSONE}
The exponent $s=q-2$ is nice if and only if $q\not\equiv 1 \pmod 6$, with the following differential multiplicities:
\vskip 3mm
\begin{center}
\begin{tabular}{ccc}
\hline
order & differential multiplicities & respective frequencies \\
\hline
\hline
$q\equiv 2 \pmod{6}$ & $0$, $2$ & $q/2$, $q/2$ \\
$q\equiv 3 \pmod{6}$ & $0$, $2$, $3$ & $(q+1)/2$, $(q-3)/2$, $1$ \\
$q\equiv 4 \pmod{6}$ & $0$, $2$, $4$ & $q/2+1$, $q/2-2$, $1$ \\
$q\equiv 5 \pmod{6}$ & $0$, $1$, $2$ & $(q-1)/2$, $1$, $(q-1)/2$\\
\hline
\end{tabular}
\end{center}
\end{proposition}
\begin{proof}
The exponent $s$ is always invertible, and odd when $p$ is odd,
so that $x^s+y^s=0$ if and only if $y=-x$, which makes $N(1,0)=0$.
The pairs $(0,1)$ and $(1,0)$ are solutions of $x+y=1$, $x^s+y^s=1$, 
and for the other contributions, we may assume $vxy\not=0$, so that
$$
	x+y=1,\quad x^s+y^s=v
\qquad 
\Longleftrightarrow
\qquad \frac 1{xy}=v,\quad x+y=1.
$$
Thus $N(1,1)$ is two plus the number of roots in $L$ of $x^2-x+1$,
and $N(1,v)$ for $v\not=0,1$ is the number of roots in $L$ of $x^2-x+1/v$.
Since $x^2-x+1$ is the sixth cyclotomic polynomial in fields of
characteristic $p> 3$ (while it degenerates to the third cyclotomic
polynomial in characteristic $2$ and to $(x+1)^2$ in characteristic $3$),
we see that
\[
N(1,1) =\begin{cases}
2 & \text{if $q \equiv 2 \pmod{3}$,} \\
3 & \text{if $q \equiv 0 \pmod{3}$,} \\
4 & \text{if $q \equiv 1 \pmod{3}$.}
\end{cases}
\]
If $v\not=0,1$ and we are in characteristic $2$, then the additive Hilbert Theorem 90 shows that $x^2-x+1/v$ has zero or two roots in $L$ depending on whether the absolute trace (from $L$ to $\Fp$) of $1/v$ is $1$ or $0$, respectively.
So if $[L:\F_2]$ is odd ($q\equiv 2\pmod{6}$), we obtain $q/2-1$ instances each of the differential multiplicities $0$ and $2$, while if $[L:\F_2]$ is even ($q \equiv 4 \pmod{6}$) we obtain $q/2$ instances of $0$ and $q/2-2$ instances of $2$.

If $v\not=0,1$ and we are in odd characteristic, the quadratic polynomial $x^2-x+1/v$ has zero, one, or two roots in $L$ depending on whether the discriminant $1-4/v$ is a quadratic nonresidue, zero, or a quadratic residue in $L$, respectively.
Our discriminant runs through all values except $1$ and $-3$.
In characteristic $3$ ($q\equiv 3\pmod{6}$), this produces $(q-1)/2$ instances differential multiplicity $0$ and $(q-3)/2$ instances of differential multiplicity $2$.
When $-3$ is a quadratic residue (i.e., when $q\equiv 1 \pmod{6}$), this produces one instance of differential multiplicity $1$, $(q-1)/2$ instances differential multiplicity $0$, and $(q-5)/2$ instances of differential multiplicity $2$.
When $-3$ is a quadratic nonresidue (i.e., when $q\equiv 5 \pmod{6}$), this produces one instance of differential multiplicity $1$, and $(q-3)/2$ instances each differential multiplicities $0$ and $2$.

We collate the information about the various $N(1,v)$ to see that we get a nice exponent if and only if $q\not\equiv 1 \pmod{6}$.
\end{proof}

It is easy to find numerically all of the differential
multiplicities for exponents over a small field 
using the Zech logarithm.   Here, we focus on odd characteristic.
Let $\omega$ be a
primitive root of the finite field $L$ of order $q$.  Then for $k\not=(q-1)/2$,  we define $\zech(k)$ to be the unique $\ell$ such that $\omega^\ell=1+\omega^k$.
The logarithm of
$$
       x^s + (1-x)^s = x^s \left( 1 + \left(\frac{1-x}x\right) ^s \right)
$$
for $x= \omega^k$ is
\[
  k \times s + \zech[ s \times ( \zech[ n + k] - k ) ].
\]
where $n = (q-1)/2$.  
In the numerical experiment \cite{NICEXP}, 
all nice exponents over fields of characteristic $2 \leq p \leq 31$ 
and order $q\leq 2^{20}$ were found.
Nice exponents $s$ with $s\not\approx 1$ occur in the majority of these fields. For example,
Table \ref{NICE} shows the nice exponents ($\not\approx 1$) 
up to the equivalence $\approx$ for the fields of order $11^m$ with $m \leq 5$. 
The sixth line of the table indicates
that $241$ is nice over $\F_{11^4}$, congruent to $1$ modulo $10$, with 
three differential multiplicities $0$, $2$, and $121$. The $s^\dag$ indicates
that $s\approx q-2$.

\begin{table}[ht!]
\caption{Nice exponents $s\not\approx 1$ over the fields $\F_{11^m}$
\label{NICE}} 
\begin{tabular}{crcrrr}
\hline
degree &$s$  &$s\pmod{p-1}$& \multicolumn{3}{r}{differential multiplicities}\\
$m$ & & &\multicolumn{3}{r}{frequency [value]}\\
\hline
\hline
1    &3  &3  &5 [0]  & 1 [1] &  5 [2]\\
     &$9^\dag$  &9  &5 [0]  & 1 [1] &  5 [2]\\
2\\
3  &3    &3 &665 [0]  & 1 [1] &665 [2]\\
   &$1209^\dag$ &9 &665 [0]  & 1 [1] &665 [2]\\
4
   &241  &1 &7380 [0] &7260 [2] &  1 [121]\\
5  
     &3  &3 &80525 [0] &  1 [1] &80525 [2]\\
&$146409^\dag$  &9 &80525 [0] &  1 [1] &80525 [2]\\
\hline
\end{tabular}
\end{table}
As one can see in this example, the nice exponents $s\not\approx 1$
are rare, and all have differential multiplicity $2$.  
Surprisingly, this seems to be a general fact.  
Indeed, for the characteristics $3$ to $31$, 
all the observed nice exponents $s\not\approx 1$ have $2$ as a 
differential multiplicity. 

\begin{problem}
All nice exponents $s\not\approx 1$ over the prime fields of characteristic
$p < 5000$ were found in the numerical experiment \cite{NICEXP}, and they all correspond to the paradigmatic 
examples of Propositions \ref{PARADIG} and \ref{MINUSONE}.  Is this a general property? 
\end{problem}

We propose some conjectures based on our 
numerical evidence.

\begin{conjecture}[nice exponent]
\label{NC}
Let $s\not\approx 1$ be an exponent over a finite field 
of odd characteristic.  If $s$ is
nice, then $2$ is a differential multiplicity.
\end{conjecture}
We can make an even stronger conjecture.
\begin{conjecture}[optimist]
Let $s\not\approx 1$ be an exponent over a finite field 
of odd characteristic. If $s$ is invertible,
then $2$ is a differential multiplicity.
\end{conjecture}
Then we claim that we have the following implications among conjectures
$$
\text{optimist}
\Longrightarrow
\text{nice exponent}
\Longrightarrow
\text{Helleseth Three-Valued}.
$$

\begin{remark}
It is interesting to notice that Conjecture \ref{NC} implies the
Helleseth Three-Valued Conjecture.
Indeed, let $s$ be an exponent over a field $L$ of characteristic $p > 3$ with $[L:\F_p]$ a power of two, and suppose that $s$ is three-valued with values $0$, $A$, and $B$.
Write $A=\alpha p^a$, $B=\beta p^b$, and $A-B=\gamma p^c$ with $p\nmid \alpha,\beta,\gamma$.
Now assume Conjecture \ref{NC} holds, so that $2$ is a differential multiplicity of $s$.
Since we must be in case \eqref{Bob} of Theorem \ref{UNIFORM}, this shows that $|\alpha\beta|=2$ and $a=b=\frac{1}{2}[L:\F_p]$.
Since the nonzero spectral values $A$ and $B$ are opposites in sign (as was seen in our proof of Theorem \ref{UNIFORM}), this means that $|A-B|=3 \sqrt{|L|}$, but since $p>3$, this contradicts the fact that $|\gamma|=1$, which also must hold in case \eqref{Bob} of Theorem \ref{UNIFORM}.
Thus, if Conjecture \ref{NC} is true, then Helleseth's Conjecture will be true in all characteristics $p > 3$.
Since Helleseth's Conjecture is already proved in characteristic $2$ and $3$ (see Corollary \ref{Richard}), it would then be fully established.
\end{remark}
\section*{Acknowledgements}
The first author was supported in part by a Research, Scholarship, and Creative Activity Award from California State University, Northridge.
The first author was also supported in part by the Institut de Math\'ematiques de Toulon at Universit\'e de Toulon as a visiting professor.
The authors thank anonymous reviewers for helpful corrections and suggestions.

\end{document}